\documentclass[11pt]{article}
\usepackage{}

\usepackage{amsmath,amsthm,verbatim,amssymb,amsfonts,amscd,graphicx}

\usepackage{bm}
\usepackage{tikz}
\usetikzlibrary{arrows,shapes,positioning}
\usetikzlibrary{decorations.markings}
\tikzstyle arrowstyle=[scale=1]
\tikzstyle directed=[postaction={decorate,decoration={markings, mark=at position 0.75 with {\arrow[arrowstyle]{stealth}}}}]
\tikzstyle redirected=[postaction={decorate,decoration={markings, mark=at position 0.35 with {\arrow[arrowstyle]{stealth}}}}]

\newtheorem{theorem}{Theorem}[section]
\newtheorem{corollary}[theorem]{Corollary}
\newtheorem{definition}[theorem]{Definition}

\newtheorem{lemma}[theorem]{Lemma}

\newtheorem{claim}{Claim}[theorem]

\DeclareMathOperator{\supp}{{\rm supp}}

\DeclareMathOperator{\Z}{\mathbb{Z}}

\newcommand{\JCTB}{{\it J. Combin. Theory Ser. B}}

\newcommand{\JGT}{{\it J. Graph Theory}}

\newcommand{\DM}{{\it Discrete Math.}}
\newcommand{\DAM}{{\it Discrete Appl. Math.}}

\newcommand{\SIAMDM}{{\it SIAM J. Discrete Math.}}

\newcommand{\CJM}{{\it Canad. J. Math.}}
\newcommand{\JLMS}{{\it J. London Math. Soc.}}

\newcommand{\EJC}{{\it European J. Combin.}}

\begin{document}

\title{Flows of   3-edge-colorable  cubic  signed graphs}
\author{\small Liangchen Li\thanks{School of Mathematical Sciences, Luoyang Normal University, Luoyang  471934, China. Email: liangchen\_li@163.com. Partially supported by NSFC (No. 11301254),  Basic Research Foundation of Henan Educational Committee (No. 20ZX004), and China Scholarship Council.},
Chong Li\thanks{Department of Mathematics, West Virginia University, Morgantown, WV 26505, USA.  Email:~ cl0081@mix.wvu.edu}, Rong Luo\thanks{Department of Mathematics, West Virginia University, Morgantown, WV 26505, USA.  Email:~ rluo@mail.wvu.edu. Partially supported by a grant from  Simons Foundation (No. 839830)},  Cun-Quan Zhang\thanks{Department of Mathematics, West Virginia University, Morgantown, WV 26505, USA.  Email:~cqzhang@mail.wvu.edu.   Partially supported
 by an  NSF grant DMS-1700218}~ and Hailiang Zhang\thanks{Department of Mathematics, Taizhou University, Taizhou 317000, China. Email: hdsdzhl@163.com}
 }
\date{}
\maketitle
\begin{abstract}
 Bouchet conjectured  in 1983 that every flow-admissible signed graph admits a nowhere-zero 6-flow which is equivalent to the restriction to cubic signed graphs.   In this paper, we proved that every flow-admissible $3$-edge-colorable cubic signed graph admits a nowhere-zero $10$-flow.  This together with the 4-color theorem implies that every  flow-admissible bridgeless  planar  signed graph admits a nowhere-zero $10$-flow. As a byproduct, we also show that every flow-admissible hamiltonian signed graph  admits a nowhere-zero $8$-flow.
\end{abstract}

\section{Introduction}
In 1983, Bouchet~\cite{Bouchet1983} proposed a flow conjecture that
{\em every flow-admissible signed graph admits a nowhere-zero $6$-flow}.
 Bouchet~\cite{Bouchet1983} himself proved that
such signed graphs admit  nowhere-zero $216$-flows and Z\'yka~\cite{Zyka1987}
further reduced to $30$-flows. Recently DeVos et al. \cite{DLLLZZ} further proved  that flow-admissible signed graphs admit  nowhere-zero $11$-flows.

Similar to  ordinary graphs, Bouchet's conjecture is equivalent to the restriction to  cubic signed graphs: {\it every flow-admissible cubic  signed graph admits a nowhere-zero $6$-flow.}

 Schubert and Steffen \cite{Schubert2015} verified Bouchet's Conjecture for Kotzig graphs. M\'a\v cajov\'a  and \v{S}koviera \cite{Mac-steffen2015} characterized  cubic signed graphs that admit a nowhere-zero $3$-flow and that admit a nowhere-zero $4$-flow, respectively.
 In this paper, we investigate integer flows in  $3$-edge-colorable cubic signed  graphs and prove the following theorem.

\begin{theorem}
\label{THM-10-flow}
Every flow-admissible  $3$-edge-colorable cubic  signed graph  admits a nowhere-zero $10$-flow.
\end{theorem}

By the $4$-color theorem, every bridgeless cubic planar graph is $3$-edge-colorable.  Therefore we have the following corollary for bridgeless signed planar graphs.
\begin{corollary}
\label{COR}
Every flow-admissible  bridgeless planar  signed graph admits a nowhere-zero $10$-flow.
\end{corollary}

Theorem~\ref{THM-10-flow}  follows from the following  stronger result which shows that every  connected  flow-admissible $3$-edge-colorable  cubic signed graph admits a nowhere-zero $8$-flow except one case which has a nowhere-zero $10$-flow.
\begin{theorem}
\label{THM-10-flow2}
Let $(G,\sigma)$ be a  connected  $3$-edge-colorable  cubic signed graph and $E_N(G,\sigma)$ be the set of negative edges in $(G, \sigma)$. Let $R, B, Y$ be   the three color classes such that $|R\cap E_N(G, \sigma)| \equiv |B\cap E_N(G, \sigma)| \pmod 2$. If $(G,\sigma)$ is flow-admissible,  then it has a nowhere-zero $8$-flow  unless $R\cup B$ contains no unbalanced  circuits and  the numbers of unbalanced circuits in  $R\cup Y$ and $B\cup Y$  are both odd and at least $3$, in which case it has a nowhere-zero $10$-flow.
\end{theorem}

As a byproduct, we also prove the following $8$-flow theorem for  hamiltonian  signed graphs.
\begin{theorem}\label{THM-9-flow}
If $(G,\sigma)$ is a flow-admissible hamiltonian  signed  graph, then $(G,\sigma)$ admits a nowhere-zero 8-flow.
\end{theorem}

The organization of the rest of the paper is as follows:  Basic notation and terminology will be introduced in Section~\ref{SEC:notation};  some lemmas needed for the proofs of the main results will be presented in  Section~\ref{SEC:lemma};  the proofs of Theorems~\ref{THM-10-flow2}  and Corollary~\ref{COR}   and the proof of Theorem~\ref{THM-9-flow} will be completed in Sections \ref{SEC:cubic}  and \ref{SEC:hamilton}, respectively.

\section{Notation and Terminology}
\label{SEC:notation}
Graphs  considered in this paper are finite and may have multiple edges or loops. For  terminology and notations not defined here we follow \cite{BM2008, Diestel2010, West1996}.

Let $G$ be a graph.  A {\em leaf vertex} is a vertex of degree $1$. For a vertex subset $U\subseteq V(G)$, denote by $\delta_G(U)$, the set of edges with one endvertex in $U$ and the other in $V(G)-U$. A path is {\em nontrivial} if it  contains at least two vertices. 
Let $u,v$ be two vertices in $V(G)$. A {\em $(u,v)$-path} is a path with $u$ and $v$ as its endvertices.
  Let $C=v_1\cdots v_r v_1$ be a circuit where $v_1,v_2,\dots, v_r$ appear in clockwise on $C$. A {\em segment} of $C$ is the path $v_i v_{i+1}\cdots v_{j-1}v_j$  contained in $C$ and is denoted by $v_iCv_j$, where the indices are taken  modulo $r$.

A {\em signed graph} $(G,\sigma)$ is a graph $G$ together with a {\em signature} $\sigma\colon\ E(G)\to\{\pm1\}$. An edge $e$ is {\em positive} if $\sigma(e) =1$ and {\em negative} if $\sigma(e) =-1$. Denote by  $E_N(G,\sigma)$ the set of negative edges in $(G,\sigma)$. Let $e = uv$ be an edge. By contracting $e$, we mean to  first identify $u$ with $v$ and  then to delete  the loop if $\sigma(e) = 1$ otherwise to  keep the negative loop. An {\em all-positive signed graph} is a  signed graph without negative edges, which is also called 
an {\em ordinary graph}.
A circuit is \emph{balanced} if it contains an even number of negative edges and  is \emph{unbalanced} otherwise. A signed graph is called {\em balanced} if it contains no unbalanced circuit. A signed graph is {\em unbalanced} if it is not balanced. A {\em signed circuit} is defined as a signed graph of any of the following three types:

(1) a balanced circuit;

(2) a short barbell, that is, the union of two unbalanced circuits that meet at a single vertex;

(3) a long barbell, that is, the union of two vertex-disjoint unbalanced circuits with a nontrivial path that meets the circuits only at its endvertices.

We regard an edge $e=uv$ of a signed graph as two half edges $h_e^u$ and $h_e^v$, where $h_e^u$ is incident with $u$ and $h_e^v$ is incident with $v$. Let $H_G(v)$ (or simply $H(v)$ if no confusion occurs) be the set of all half edges incident with $v$, and $H(G)$  be the set of all half edges of $(G,\sigma)$. An {\em orientation} of $(G,\sigma)$ is a mapping $\tau : H(G)\rightarrow \{+1,-1\}$ such that for each $e=uv\in E(G)$, $\tau (h_e^u)\tau (h_e^v)=-\sigma(e)$. For $h_e^u\in H(G)$, $h_e^u$ is {\em oriented away from $u$} if $\tau (h_e^u)=1$ and $h_e^u$ is {\em oriented toward $u$} if $\tau (h_e^u)=-1$. A signed graph $(G,\sigma)$ together with an orientation $\tau$ is called an {\em oriented signed graph}, denoted by $(G,\tau)$.

\begin{definition}
\label{DE: Flow}
Let $(G,\sigma)$ be a signed graph  and $\tau$ be an orientation of $(G, \sigma)$. Let $k\geq 2$ be an integer and  $f: E(G) \to \Z$ be a mapping such that $|f(e)| \leq k-1$.
\begin{itemize}
\item[(1)] For each vertex $v$,  the {\em boundary} of $(\tau,f)$ at $v$ is $\partial (\tau, f)(v)=\sum_{h\in H(v)}\tau(h)f(e_h)$.

\item[(2)]The {\em support} of $f$, denoted by $\supp (f)$, is the set of edges $e$ with $|f(e)| >0$.

\item [(3)] If $\partial (\tau,f)(v)=0$ for each vertex $v$, then $(\tau,f)$  is called a {\em $k$-flow} of $(G,\sigma)$.  A $k$-flow  $(\tau,f)$ is said to be  {\em nowhere-zero}  if $\supp (f)=E(G)$.

\item [(4)]   If $\partial (\tau,f)(v) \equiv 0 \pmod k$ for each vertex $v$,
then  $(\tau,f)$ is called a {\em $\Z_k$-flow} of $(G,\sigma)$.  A $\Z_k$-flow $(\tau,f)$ is said to be  {\em nowhere-zero}  if $\supp (f)=E(G)$.
\end{itemize}
\end{definition}

For a mapping $f: E(G)\to {\mathbb Z}$, denote $E_{f=\pm i}=\{e\in E(G) : |f(e)|=i\}.$

For convenience, we  shorten the notation of nowhere-zero  $k$-flow  and nowhere-zero $\Z_k$-flow as $k$-NZF and $\Z_k$-NZF,  respectively.  If the orientation is understood from the context, we use $f$  instead of $(\tau, f)$ to denote a flow.

A signed graph is {\em flow-admissible} if it admits a nowhere-zero $k$-flow for some integer $k$. In a signed graph,
 {\em switching}  a vertex $u$ means reversing the signs of all edges incident with $u$. Two signed graphs are {\em equivalent} if one can be obtained from the other
by a sequence of switches. Note that  a signed graph is  balanced if and only  if it is equivalent to an all-positive graph.
 Note that switching  a vertex does not change the parity of the number of negative edges in a circuit and  although technically it  changes the flows, it only reverses the directions of the half edges incident with the vertex, but the directions of other half edges and the flow values of all edges remain the same. Bouchet \cite{Bouchet1983} gave a characterization for flow-admissible signed graphs.

\begin{lemma}{\rm (Bouchet~\cite{Bouchet1983})}
\label{Flow-amissible}
A connected signed graph $(G,\sigma)$ is flow-admissible if and only if it is not equivalent to a signed graph with exactly one negative edge and it has no bridge $b$ such that $(G-b,\sigma|_{G-b})$ has a balanced component.
\end{lemma}

\section{Lemmas}
\label{SEC:lemma}

Given a signed graph $(G,\sigma)$, 
let $H$ be a signed subgraph of $(G, \sigma)$ and $C$ be a balanced circuit. Define the following operation:
$$\Phi_2\mbox{-operation}: \mbox{ add a balanced circuit $C$ into $H$ if $|E(C)- E(H)|\leq 2$.}$$
We use $\langle H\rangle_2$ to denote the maximal
 subgraph of $G$ obtained from $H$ via $\Phi_2$-operations. Z\'{y}ka \cite{Zyka1987} proved the following result.

\begin{lemma}{ \rm (Z\'{y}ka \cite{Zyka1987})}
\label{2-operation}
Let $(G,\sigma)$ be a signed graph and $H$ be a subgraph of $G$. If  $\langle H\rangle_2=G$, then $(G,\sigma)$ admits a $\Z_3$-flow $\phi$ such that $E(G) - E(H)\subseteq \supp(\phi)$.
\end{lemma}

It is clear that a signed graph admits a $\Z_2$-NZF if and only if each component of $(G,\sigma)$ is eulerian.
The next lemma gives a characterization of signed graphs admitting a $2$-NZF.
\begin{lemma}{\rm (Xu and Zhang \cite{Xu2005})}
 \label{eulerian-2-flow}
A signed graph $(G,\sigma)$ admits a $2$-NZF if and only if each component of $(G,\sigma)$ is eulerian and has an even number of negative edges.
\end{lemma}

The next two lemmas show  how to  convert a modulo flow to an integer-valued flow.

\begin{lemma}{\rm (Cheng et al. \cite{CLLZ2018})}
\label{TH: 3-to-4}
Let $(G,\sigma)$ be a bridgeless signed graph. If $(G,\sigma)$ admits a  $\Z_3$-flow $f_1$, then it admits a  $4$-flow $f_2$ such that $\supp(f_1)\subseteq  E_{f_2 = \pm 1} \cup E_{f_2 = \pm 2}$.
\end{lemma}

\begin{definition}
Let $f$ be a $\Z_2$-flow of $(G,\sigma)$. Then $\supp(f)$  consists of  vertex-disjoint  eulerian subgraphs. A component of $\supp(f)$ is called {\em even} if it contains an even number of negative edges; otherwise it is called {\em odd}.
\end{definition}

\begin{lemma} {\rm (Cheng et al. \cite{CLLZ2018})}
\label{TH: 2-to-3} Let $(G,\sigma)$ be a connected signed graph. If
 $(G,\sigma)$ admits a $\Z_2$-flow $f_1$
such that $\supp(f_1)$ contains an even number of odd components,
then it admits a $3$-flow $f_2$
such that $\supp(f_1)=E_{f_2 = \pm 1}$ and $\supp(f_2)/\supp(f_1)$ is acyclic.
\end{lemma}

Lemma~\ref{TH: 2-to-3} can be extended to the case when the support of a $\Z_2$-flow contains an odd number of odd components in the following lemma.
\begin{lemma}
\label{flow of g-barbell}
 Let $(G,\sigma)$ be a connected signed graph. If
 $(G,\sigma)$ admits a $\Z_2$-flow $f_1$
such that  the number of  odd components of $\supp(f_1)$  is odd and is at least three, then $(G,\sigma)$ has a $5$-flow  $f_2$ satisfying

\medskip \noindent
(1)  $\supp(f_2)/\supp(f_1)$ is acyclic;

\medskip \noindent
(2)  $\supp(f_1)\subseteq \{e\in E(G)\colon\ 1\leq |f_2(e)| \leq 3\}$ and $|f_2(e)| \in \{1,2\}$ for each negative loop $e \in \supp(f_1)$.
\end{lemma}
\begin{proof}
Let $(G,\sigma)$ together with a $\Z_2$-flow $(\tau, f_{1})$
be a counterexample to Lemma~\ref{flow of g-barbell}
such that $|E(G)|$ is minimized. In the  following, we always assume the flows are under the orientation $\tau$ or its restriction on according subgraphs.

Denote by $\mathcal{B}$ the set of components of $\supp(f_{1})$ and  let $H= G/\supp(f_1)$. Thus $V(H)$ can be partitioned into three parts: $X, Y$ and $W$ where $X$ and $Y$ are the sets of vertices corresponding to even and odd components in $\mathcal{B}$ respectively and $W$ is corresponding to the vertices  which are also the vertices in $V(G)$. For $u \in X\cup Y$, let $B_u$ denote the corresponding component in $\mathcal{B}$.

\begin{claim}
\label{CL: cut-edge}
$G$ contains no leaf vertices and  $H$ is a tree.
\end{claim}
\begin{proof}
If $G$ contains a leaf vertex, say $x$, then $f_1(e) = 0$ where $e$ is the edge incident with $x$ and $G-x$ remains connected. This contradicts to the minimality of $G$.

Clearly $H$ is connected since $G$ is connected. If $H$ is not a tree, then there is an edge $e \in E(G)$ such that $f_1(e) = 0$ and $G-e$ is connected, a contradiction to the minimality of $G$ again.
\end{proof}

Let $u$ be a leaf vertex of $H$ and $v$ be its neighbor.  By Claim~\ref{CL: cut-edge}, $u \in X\cup Y$.  Since $u$ is a leaf vertex of $H$, there is only one edge in $G$ with one endvertex in $B_u$ and the other one in $B_v$. Let $x_ux_v$ be the only edge  in $G$ where $x_u \in V(B_u)$ and $x_v \in V(B_v)$. 

\begin{claim}
\label{CL:leaf}
 $u\in Y$ and  $v \not \in Y$.
\end{claim}
\begin{proof}
Suppose to the contrary that  either $u \in X$ or $u\in Y$ and $v \in Y$.   Let $G' = G- V(B_u)$.  Since $B_u$ is a leaf block, $G'$ is connected.

If $u\in X$, then  $B_u$ is an even component and thus  $\mathcal{B} -B_u$ and $\mathcal{B}$ have  the same number of odd components. Since $G'$ is connected, by the minimality of $G$, there is an integer $5$-flow $g_1$ of $(G', \sigma |_{E(G')})$ such that $\supp(f_1)-E(B_u)\subseteq \supp(g_1)$ and $g_1$ satisfies (1) and (2).  Then $g_1$ can be considered as a flow of $(G,\sigma)$ under the orientation $\tau$ such that $E(B_u)\cap \supp(g_1) = \emptyset$. Since $B_u$ is an even eulerian component, by Lemma~\ref{eulerian-2-flow}, there is a $2$-flow $g_2$ of $(G,\sigma)$ such that  $\supp(g_2) = E(B_u)$. Therefore $g_1+g_2$  is  a $5$-flow of $(G,\sigma)$ satisfying (1) and (2), a contradiction.

Now assume that $u \in Y$ and $v\in Y$. Then $\mathcal{B} -B_u$ has an even number of odd components.   By Lemma~\ref{TH: 2-to-3}, there is a $3$-flow  $g_3$ such that $\supp(f_1)-E(B_u)\subseteq \supp(g_3)$ and  $g_3$ satisfies (1) and (2). Note that $g_3(x_ux_v) = 0$ and $g_3(e) = 0$ for each $e\in E(B_u)$.    By Lemma~\ref{TH: 2-to-3} again, there is a $3$-flow $g_4$ such that $\supp(g_4) = E(B_u)\cup E(B_v) + x_ux_v$,  $E(B_u) \cup E(B_v) = E_{g_4 = \pm 1}$, and $\{x_ux_v\} = E_{g_4 = \pm 2}$.  Therefore $g_3 + 2g_4$ is a desired $5$-flow, a contradiction. This proves the claim.
\end{proof}

Let $(G_1,\sigma_1)$ be the signed graph obtained from $G - V(B_u)$ by adding a negative loop $e_1$ at $x_v$ where $\sigma_1$ is defined as $\sigma_1(e) = \sigma(e)$ for each $e\in E(G_1)-\{e_1\}$ and $\sigma_1(e_1) = -1$. The orientation $\tau_1$ of $(G_1,\sigma_1)$ is defined as $\tau_1 (h) = \tau(h)$ for each $h \in H(G_1)$ and $h$ is not an half edge of the loop $e_1$; for each half edge $h$ of $e_1$, $\tau_1(h) = \tau(h_{uv}^v)$.  

Let $(G_2,\sigma_2)$ be the signed graph  obtained from $B_u$ by adding a negative loop $e_2$ at $x_u$. Its signature $\sigma_2$  and orientation $\tau_2$ are  defined similarly to $\sigma_1$ and $\tau_1$, respectively.  

Denote   $B_v' = B_v\cup \{e_1\}$ and $\mathcal{B'} = \mathcal{B} - B_u-B_v + B_v'$ if $v \in X\cup Y$; otherwise denote $B_v' = \{e_1\}$  and  $\mathcal{B'} = \mathcal{B} - B_u + B_v'$.    Note that   there is a $\Z_2$-flow of $(G_1,\sigma_1)$ whose support is $\bigcup_{B\in \mathcal{B'}}E(B)$.

   By Claim~\ref{CL:leaf},  both $B_u$ and  $B_v'$  are odd. Thus  $\mathcal{B'} $ and  $\mathcal{B}$ have the same number of odd components.  By the minimality of $G$, there is  a 
 $5$-flow $(\tau_1,g_5)$  of $(G_1, \sigma_1)$    satisfying (1) and (2).

By Claim~\ref{CL:leaf},    $(G_2,\sigma_2)$ is a signed eulerian graph with even number of negative edges. By Lemma \ref{eulerian-2-flow}, there is a  $2$-flow $(\tau_2,g_6)$  of $(G_2,\sigma_2)$ such that $\supp(g_6) = E(G_2)$.  We may assume $g_5(e_1)g_6(e_2) > 0$ otherwise replacing $g_6$ with $-g_6$.  Let $a = g_5(e_1)$. Then $|a| \in \{1,2\}$.

Let $(\tau, g_7)$ be  the integer flow  of $(G,\sigma)$ defined as follows: for each $e \in E(G)$,
$$g_7(e) =  \left\{
\begin{array}{ll}
g_5(e) & \mbox{ if } e\in \supp(g_5);\\
ag_6(e) & \mbox{ if } e\in \supp(g_6);\\
2a & \mbox{ if } e = uv;\\
0       & \mbox{ otherwise. }
\end{array}
\right.
$$

Then $g_7$ is a $5$-flow  of $(G,\sigma)$ satisfying (1) and (2), a contradiction. This completes the proof of the lemma.
\end{proof}

The following lemma is due to Zaslavsky  \cite{Zas1982}.

\begin{lemma}{\rm (Zaslavsky  \cite{Zas1982})}
\label{balanced-extension}
Let $T$ be a spanning tree of a signed graph $(G, \sigma)$. For every $e\notin E(T)$, let $C_e$ be the unique circuit contained in $T+e$. If the circuit $C_e$ is balanced for every $e\notin E(T)$, then $G$ is balanced.
\end{lemma}

  The proof of the following lemma is inspired by the proof of Theorem 4.2  in \cite{MS-eulerian} due to  M\'{a}\v{c}ajov\'{a} and \v{S}koviera.

\begin{lemma}\label{HC-cover}
Let $C$ be an unbalanced circuit of a signed graph $(G,\sigma)$.
If $(G,\sigma)$ is flow-admissible and $G-E(C)$ is balanced, then $(G,\sigma)$ has a $4$-flow $f$ satisfying the following:

(1) $E(C) \subseteq \supp(f)$;

(2) In $H = G[\supp(f)]$ the subgraph induced by $\supp(f)$,   each vertex in $V(H)-V(C)$  has degree at most $3$ in $H$ and at most one vertex in $V(H)-V(C)$ has degree $3$.
\end{lemma}

\begin{proof}
Denote by $G'=G-E(C)$. Since $G'$ is balanced, with some switching operations, we may assume that all edges in $E(G')$ are positive and thus $E_N(G, \sigma)\subseteq E(C)$. Fix an orientation $\tau$ of $(G,\sigma)$ and  in the following we always assume the flows are under the orientation $\tau$ or its restriction on according subgraphs.

Let $M$ be a component of $G'$. The circuit $C$ is divided by the vertices of $M$ into segments whose endvertices lie in $M$ and all inner vertices lie outside $M$. An endvertex of a segment is called an \emph{attachment} of $M$.  A segment is called positive (negative) if it contains an even  (odd) number of negative edges.  Let $S$ be a segment. Note that $M\cup S$ is unbalanced (balanced) if and only if the segment $S$ is negative (positive). Since $C$ is unbalanced, the number of negative segments determined by each component $M$  is odd.

We prove  the lemma by contradiction. Suppose to the contrary that $(G,\sigma)$ has no $4$-flow satisfying (1) and (2).

\begin{claim}
Each component  of $G'$  determines exactly one negative segment.
\end{claim}
\begin{proof}
 Suppose to the contrary  that  $M$ determines  more than one negative segments.  Thus $M$ determines at least three negative segments. Let $u_1Cu_1'$, $u_2Cu_2'$, $u_3Cu_3'$ be three consecutive  negative segments (in clockwise) where $u_i$ and $u_i'$ are attachments for $i=1,2,3$. Then $u_1'Cu_2$, $u_2'Cu_3$, $u_3'Cu_1$ all contain even number of negative edges.  This implies that  $C$ can be partitioned into three  negative segments: $u_1Cu_2$, $u_2Cu_3$, and $u_3Cu_1$.

 We first show that no  $(u_1,u_2)$-path in $M$  passes through $u_3$.  Otherwise let $P$ be a $(u_1,u_2)$-path in $M$ that passes through $u_3$.  Then $C_1= u_1Cu_3+ u_1Pu_3$ and $C_2= u_3Cu_2+ u_3Pu_2$ both are balanced circuits.  By Lemma~\ref{eulerian-2-flow}, there  is a $2$-flow $f_i$ of $(G,\sigma)$  such that $\supp(f_i) = E(C_i)$ for each $i = 1,2$. Therefore  $2f_1+f_2$ is a $4$-flow of $(G,\sigma)$ and $\supp(2f_1 + f_2) = E(C)\cup E(P)$, which  is a desired $4$-flow, a contradiction.

 By symmetry, no $(u_i,u_j)$-path passes through $u_k$ where $\{i,j,k\} = \{1,2,3\}$.  This implies  that $u_1$ and $u_2$ are not adjacent.  Otherwise, a $(u_1,u_3)$-path together with $u_1u_2$ gives a $(u_2,u_3)$-path containing $u_1$.

 Let $P_1$ be a $(u_1,u_2)$-path. Since $M$ is connected, there is a path $P_2$ from $u_3$ to $P_1$ such that $|V(P_2)\cap V(P_1)| =1$. Let $v$ be the only common vertex in $P_1$ and $P_2$.  Then  $C,  P_1$, and $P_2$  form a  signed graph as illustrated in  Figure \ref{FIG: 1} which has a   desired $4$-NZF, a contradiction again. This completes  the proof of the claim.
 \end{proof}

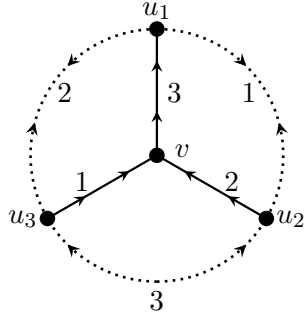
\begin{figure}
\begin{center}
\begin{tikzpicture}[scale=1.2]
\path (0:0)   coordinate (v);  \draw [fill=black] (v) circle (0.08cm);
\path (-30:1.4) coordinate (u2); \draw [fill=black] (u2) circle (0.08cm);
\path (90:1.4) coordinate (u1); \draw [fill=black] (u1) circle (0.08cm);
\path (-150:1.4) coordinate (u1'); \draw [fill=black] (u1') circle (0.08cm);
\path (-90:1.4) coordinate (w);
\draw [redirected, directed,line width=0.85] (v)--(u1);
\draw [redirected, directed,line width=0.85] (u2)--(v);
\draw [redirected, directed,line width=0.85] (u1')--(v);

\draw [directed, dotted, line width=1]  (u2) arc(-30:30:1.4);
\draw [directed, dotted, line width=1]  (u1) arc(90:30:1.4);
\draw [directed, dotted, line width=1]  (u1) arc(90:150:1.4);
\draw [directed, dotted, line width=1]  (u1') arc(210:150:1.4);
\draw [directed, dotted, line width=1]  (w) arc(-90:-30:1.4);
\draw [directed, dotted, line width=1]  (w) arc(-90:-150:1.4);

\node[right]at (30:0.1) {$v$};
\node[above]at (90:1.4) {$u_1$};
\node[right]at (-30:1.4) {$u_2$};
\node[left] at (-150:1.4) {$u_3$};

\node[right]at (90:0.7) {$3$};\node[right]at (-25:0.7) {$2$};
\node[left]at (-155:0.7) {$1$};\node[below]at (-90:1.4) {$3$};
\node[left]at (30:1.4) {$1$};\node[right]at (150:1.4) {$2$};
\end{tikzpicture}
\end{center}
\caption{\small a 4-flow covers $C$}
\label{FIG: 1}
\end{figure}

Let $\mathcal{M}$ denote the set of all components of $G'$. For each component $M$, denote by $S_M=uCv$ the negative segment determined by $M$ where $u$ and $v$ are two attachments of $M$ on $C$. Denote by $S_M' = vCu$ the cosegment  of $S_M$. Then $E(S_M) \not = \emptyset$ and $E(S_M' ) = E(C) - E(S_M)$.

\begin{claim}
$\bigcap_{M\in \mathcal{M}} E(S_M)=\emptyset$. Therefore  $\bigcup _{M\in \mathcal{M}} E(S'_M)= C$ and $|\mathcal{M}|\geq 2$.
\end{claim}
\begin{proof}
Suppose to the contrary  $\bigcap_{M\in \mathcal{M}} E(S_M) \not =\emptyset$. Let $e^*\in\bigcap_{M\in \mathcal{M}} E(S_M)$.  Then there is a spanning tree $T$ of $G-e^*$ containing the path $P^*=C-e^*$.  Let $e=uv \in E(G)-e^*-E(T)$.  Denote the unique circuit contained in $T+e$ by $C_e$.

 If $E(C_e)\cap E(P^*)=\emptyset$, then $C_e$  contains no negative edges and thus  is balanced.

 Assume that $C_e$ and $P^*$ have common edges. Since $T$ contains all the edges in $C-e^*$,  $E(C_e)\cap E(C)$ is a path  $P$ on $C$. Let $u'$ and $v'$ be the two endvertices of $P$ in clockwise order on $C$. Then $C_e[(V(C_e)- V(P)) \cup \{u',v'\}]$ is a also a path and thus it is contained in some component $M\in \mathcal{M}$.  This implies that $u'$ and $v'$ are two attachments of $M$ on $C$.  Since $e^*$ belongs to the only negative segment of $C$ determined by $M$, $u'Cv'$ is the union of some positive segments of $C$ determined by $M$.  Therefore $C_e$ has an even number of negative edges and thus is balanced.
  By Lemma \ref{balanced-extension}, $G-e^*$ is balanced, contradicting Lemma \ref{Flow-amissible}. This proves $\bigcap_{M\in \mathcal{M}} E(S_M)=\emptyset$.

Since  $E(S_M') = E(C) - E(S_M)$ and  $\bigcap_{M\in \mathcal{M}} E(S_M)=\emptyset$, we have $\bigcup_{M\in \mathcal{M}} E(S_M')= C$.

Since $E(S_M) \not = \emptyset$ and $\bigcap_{M\in \mathcal{M}} E(S_M)=\emptyset$, we have  $|\mathcal{M}| \geq 2$.
\end{proof}

Let $\mathcal {S}=\{S'_1,S'_2,\ldots,S'_t\}$ be a minimal cosegment cover of $C$. Then $S_i'\not \subseteq S_j'$ for any $i, j$. 

\begin{claim}
\label{Mini-cover}
(i) For each pair $i,j\in \{1,2,\ldots,t\}$, either $S'_i\cap S'_j$   consists  of some nontrivial paths or $S_i'$ and $S_j'$ are vertex-disjoint;

(ii) Each  edge $e\in E(C)$ is contained in  at most two cosegments.
\end{claim}
\begin{proof}
(i) \ Note that  for any two segments $S_i$ and $S_j$,  their endvertices belong to two vertex-disjoint components $M_i$ and $M_j$.    Thus  no component of  $S_i' \cap S'_j$ is an isolated vertex. This proves (i).

(ii) \ Suppose to the contrary that there is an edge $e = uv$ that belongs to three cosegments, say $S_1'$, $S_2'$, $S_3'$. Let $S_i'=u_iCv_i$ for each $i \in \{1,2,3\}$.   Without loss of generality, we may assume that $u_1,u_2,u_3,u,v$ appear in this clockwise cyclic order. Then there exists a pair $i, j$ such that $u_iCv_j$ contains all the $u_l,v_l's$ ($l\in \{1,2,3\}$) and $u,v$. Hence, there is a $k \in \{1,2,3\}-\{i,j\}$ such that $u_k$ and $v_k$ are properly included in $u_iCv_j$. In this case, either $\mathcal {S}\setminus \{S_k'\}$ is still a cover of $C$ or $S_k'\cup S_i' = C$, both in contradiction to the minimality of $\mathcal {S}$.
 This completes the proof of the claim.
\end{proof}

\noindent
{\bf The final step.}  For each $i=1,\dots, t$,  denote by $S_i'=x_iCy_i$ and let $P_i$ be a path in $M_i$  connecting $x_i$ and $y_i$. Then $C_i = S_i'\cup P_i$ is a balanced eulerian subgraph.  By Claim~\ref{Mini-cover}, we may assume that  the vertices $x_1, y_t,  x_2, y_1, \dots, x_t, y_{t-1}, x_1$ appear on $C$ in the acyclic order.
 Then  $C_i\cap C_j \not = \emptyset$ if and only if  $|j-i|  \equiv  1 \pmod t$. Moreover  $C_i\cap C_{i+1}  = x_{i+1}Cy_i$ where the subindices are taken modulo $t$. See Figure~\ref{FIG: 3}  for an illustration with $t = 5$.

\begin{figure}
\begin{center}
\begin{tikzpicture}[scale=1.2]
\path (180:1.5)   coordinate (x5);  \draw [fill=black] (x5) circle (0.06cm);
\path (144:1.5)   coordinate (y4);  \draw [fill=black] (y4) circle (0.06cm);
\path (108:1.5)   coordinate (x1);  \draw [fill=black] (x1) circle (0.06cm);
\path (72:1.5)   coordinate (y5);  \draw [fill=black] (y5) circle (0.06cm);
\path (36:1.5)   coordinate (x2);  \draw [fill=black] (x2) circle (0.06cm);
\path (0:1.5)   coordinate (y1);  \draw [fill=black] (y1) circle (0.06cm);
\path (-36:1.5)   coordinate (x3);  \draw [fill=black] (x3) circle (0.06cm);
\path (-72:1.5)   coordinate (y2);  \draw [fill=black] (y2) circle (0.06cm);
\path (-108:1.5)   coordinate (x4);  \draw [fill=black] (x4) circle (0.06cm);
\path (-144:1.5)   coordinate (y3);  \draw [fill=black] (y3) circle (0.06cm);
\path (0:1.5)   coordinate (v);
\draw [ line width=1]  (v) arc(0:360:1.5);
\draw [ line width=1]  (x1) arc(-180:-72:1.5);
\draw [ line width=1]  (x2) arc(108:216:1.5);
\draw [ line width=1]  (x3) arc(36:144:1.5);
\draw [ line width=1]  (x4) arc(-36:72:1.5);
\draw [ line width=1]  (x5) arc(-108:0:1.5);
\node[above]at (108:1.5) {$x_1$};
\node[right]at (36:1.5) {$x_2$};
\node[right]at (-36:1.5) {$x_3$};
\node[below]at (-108:1.5) {$x_4$};
\node[left]at (-180:1.5) {$x_5$};

\node[below]at (-90:2.8) {\small (a) a minimal cosegment cover with $t = 5$};
\node[right]at (0:1.5) {$y_1$};
\node[below]at (-72:1.5) {$y_2$};
\node[left]at (-144:1.5) {$y_3$};
\node[left]at (144:1.5) {$y_4$};
\node[right]at (78:1.65) {$y_5$};

\node at (0:2.7){$\Longrightarrow$};

\path (0:3.5)   coordinate (w);
\path (3.8,0.7) coordinate (y5);\draw [fill=black] (y5) circle (0.06cm);
\path (3.8,-0.7) coordinate (x5);\draw [fill=black] (x5) circle (0.06cm);
\path (5.2,0.7) coordinate (x1);\draw [fill=black] (x1) circle (0.06cm);
\path (5.2,-0.7) coordinate (y4);\draw [fill=black] (y4) circle (0.06cm);
\draw [directed,redirected, line width=1] (x1) arc(45:-45:1);
\draw [directed,redirected, line width=1] (y4) arc(-45:-135:1);
\draw [directed,redirected, line width=1] (x5) arc(-135:-225:1);
\draw [directed,redirected, line width=1] (y5) arc(135:45:1);

\node[above] at (4.5,1){$1$};\node[below] at (4.5,-1){$1$};
\node[left] at (3.5,0){$1$};\node[right] at (5.5,0){$1$};
\node at (4.5,0){$C_5$};
\node[below] at (4.5,-2.8){\small (b) 2-flow $f_5$ of $C_5$};

\node[left]at (3.8,0.8) {$y_5$};\node[left]at (3.8,-0.8) {$x_5$};
\node[right]at (5.2,0.8) {$x_1$};\node[right]at (5.2,-0.8) {$y_4$};
\path (7,0) coordinate (v2);\draw [fill=black] (v2) circle (0.06cm);
\path (8,0) coordinate (u3);\draw [fill=black] (u3) circle (0.06cm);
\path (8,1) coordinate (v1);\draw [fill=black] (v1) circle (0.06cm);
\path (7,1) coordinate (u2);\draw [fill=black] (u2) circle (0.06cm);
\path (7,2) coordinate (v5);\draw [fill=black] (v5) circle (0.06cm);
\path (8,2) coordinate (u1);\draw [fill=black] (u1) circle (0.06cm);
\path (8,-1) coordinate (v3);\draw [fill=black] (v3) circle (0.06cm);
\path (7,-1) coordinate (u4);\draw [fill=black] (u4) circle (0.06cm);
\path (7,-2) coordinate (v4);\draw [fill=black] (v4) circle (0.06cm);
\path (8,-2) coordinate (u5);\draw [fill=black] (u5) circle (0.06cm);
\draw [directed, redirected,line width=1]  (u3) arc(-30:30:1);
\draw [directed, redirected,line width=1]  (u3) arc(30:-30:1);

\draw [directed, redirected, line width=1]  (u4) arc(210:150:1);
\draw [directed, redirected, line width=1]  (u2) arc(150:210:1);
\draw [directed, redirected, line width=1]  (u4) arc(150:210:1);
\draw [directed, redirected, line width=1]  (u5) arc(-30:30:1);
\draw [directed, redirected, line width=1]  (u2) arc(210:150:1);
\draw [directed, redirected, line width=1]  (u1) arc(30:-30:1);
\draw [directed, redirected, line width=1]  (v2)--(u3);
\draw [directed, redirected, line width=1]  (v1)--(u2);
\draw [directed, redirected, line width=1]  (v5)--(u1);
\draw [directed, redirected, line width=1]  (v3)--(u4);
\draw [directed, redirected, line width=1]  (v4)--(u5);
\node[left]at (7,0) {$y_2$};
\node[left]at (7,1) {$x_2$};
\node[left]at (7,2) {$y_5$};
\node[left]at (7,-1) {$x_4$};
\node[left]at (7,-2) {$y_4$};
\node[right]at (8,0) {$x_3$};
\node[right]at (8,1) {$y_1$};
\node[right]at (8,2) {$x_1$};
\node[right]at (8,-1) {$y_3$};
\node[right]at (8,-2) {$x_5$};
\node at (7.5,0.6) {$C_2$};\node at (7.5,1.6) {$C_1$}; \node at (7.5,-0.4) {$C_3$};
\node at (7.5,-1.4) {$C_4$};
\node[right]at (8.1,0.5) {\small $1$};\node[right]at (8.1,1.5) {\small $1$};
\node[right]at (8.1,-0.5) {\small $1$};\node[right]at (8.1,-1.5) {\small $1$};
\node[left]at (6.9,0.5) {\small $1$};\node[left]at (6.9,1.5) {\small $1$};
\node[left]at (6.9,-0.5) {\small $1$};\node[left]at (6.9,-1.5) {\small $1$};
\node[above]at (7.5,2) {\small $1$};\node[below]at (7.5,-2) {\small $1$};
\node[above]at (7.5,0) {\small $1+1$};\node[above]at (7.5,1) {\small $1+1$};
\node[above]at (7.5,-1) {\small $1+1$};
\node[below]at (7.5,-2.8) {\small (c) 2-flows $f_1,f_2,f_3,f_4$ };
\end{tikzpicture}

\end{center}
\caption{\small Minimum cosegment cover  and $4$-flow}
\label{FIG: 3}
\end{figure}
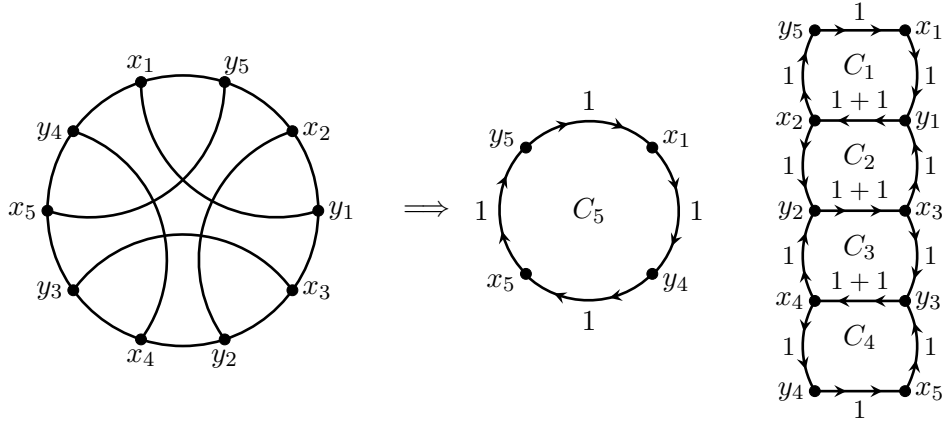

 For each $i \in \{1,2,\dots, t\}$, let $(\tau, f_i)$ be a $2$-flow  of $(G,\sigma)$  such that  $\supp(f_i) = E(C_i)$. We may assume that  for each $i =1,\dots, t-1$,  $f_i(e) = f_{i+1}(e)$ for each $e \in E(C_i)\cap E(C_{i+1})$. Then $\phi = \sum_{i=1}^{t-1}f_i + 2f_t$ is a $4$-flow  of $(G,\sigma)$ satisfying  $E(C) \subseteq \supp(\phi)= E(C)\cup [\cup_{i=1}^t E(P_i)]$.  Since $P_1,\dots, P_t$ belong to different components of $G'$, they are pairwise vertex-disjoint. Thus for each vertex $v \in V(\supp(\phi))-V(C)$, the degree of $v$ in $\supp(\phi)$ is two. Therefore $\phi$ is a  $4$-flow satisfying (1) and (2), a contradiction to the assumption that $(G,\sigma)$ is a counterexample.  This  contradiction completes the proof of the lemma.
\end{proof}

The proof of the following lemma is straightforward and thus  is omitted.
\begin{lemma}
\label{LE: flow-extension}
Let $(G,\sigma)$ be a signed graph  and $C$ be a chordless circuit whose edges are all positive. Suppose that $2\leq |\delta(V(C))| \leq 3$ and $k\geq 4$ is an integer. If $(G/C,\sigma)$  has a $k$-NZF $f$, then $f$ can be extended to be a $k$-NZF of $(G,\sigma)$.
\end{lemma}

\section{Proofs of Theorem \ref{THM-10-flow2} and Corollary~\ref{COR}}
\label{SEC:cubic}
Let's first recall  Theorem \ref{THM-10-flow2}.

\medskip \noindent
{\bf Theorem \ref{THM-10-flow2}}
{\it Let $(G,\sigma)$ be a  connected  $3$-edge-colorable cubic signed graph and $E_N(G, \sigma)$ be the set of negative edges in $(G,\sigma)$. Let $R, B, Y$ be  the three color classes such that $|R\cap E_N(G, \sigma)| \equiv |B\cap E_N(G,  \sigma)| \pmod 2$. If $(G,\sigma)$ is flow-admissible,  then it has a nowhere-zero $8$-flow  unless $R\cup B$ contains no unbalanced  circuits and  the numbers of unbalanced circuits in  $R\cup Y$ and $B\cup Y$  are both odd and at least $3$, in which case it has a nowhere-zero $10$-flow.  }

\begin{proof} Let $\tau$ be an orientation of $(G,\sigma)$. In the following we always assume the flows are under the orientation $\tau$ or its restriction on according subgraphs.
 Denote by $M_1M_2$ the $2$-factor induced by $M_1\cup M_2$ for each pair $M_1, M_2\in\{R, B, Y\}$. Since  $|R\cap E_N(G, \sigma)| \equiv |B\cap E_N(G, \sigma)| \pmod 2$, $RB$ has an even number of odd components.

\medskip \noindent
{\bf Case 1.} $RB$ contains an unbalanced circuit.

Then  by Lemma~\ref{TH: 2-to-3},  $(G,\sigma)$ has a $3$-flow  $(\tau, f_1)$ such that $RB = E_{f_1 = \pm 1}$ and $|f_1(e)| =  2$ only if $e \in Y$.

\medskip \noindent
{\bf Subcase 1.1.} $|Y\cap E_N(G, \sigma)|\equiv |R\cap E_N(G,\sigma)|\equiv|B\cap E_N(G,\sigma)|\pmod 2$.

Then  $RY$ has an even number of unbalanced circuits.
  By Lemma~\ref{TH: 2-to-3}  again,  $(G,\sigma)$ has a $3$-flow  $(\tau f_2)$ such that $RY = E_{f_2 = \pm 1}$ and $|f_2(e)| =  2$ only if $e \in B$.

Then $f = f_1 + 3f_2$  is a  $9$-NZF of $(G,\sigma)$. Since $E_{f_2= \pm 2} \cap E_{f_1 = \pm 2} = \emptyset$, $|f(e)| \not =  8$. Thus $f$ is indeed an $8$-NZF of $(G,\sigma)$.

\medskip \noindent
{\bf Subcase 1.2.} $RY$ or $BY$ has an odd number of unbalanced circuits. 

In this case, both  $RY$ and $BY$ have an odd number of unbalanced circuits.

Let $C$ be an unbalanced circuit in $RB$. Let  $R'=R\bigtriangleup C$ and $B'=B\bigtriangleup C$ with $R$ and $B$, respectively (this is equivalent to swap colors $R$ and $B$ on $C$). This implies that $|Y\cap E_N(G, \sigma)|\equiv |R'\cap E_N(G, \sigma)|\equiv|B'\cap E_N(G, \sigma)|\pmod 2$. We are back to Subcase 1.1.

\medskip \noindent
{\bf Case 2.} $RB$ contains  no  unbalanced circuit.

 Then by Lemma~\ref{eulerian-2-flow},  $(G,\sigma)$ has a $2$-flow  $f_3$ such that $\supp(f_3) = RB$.

\medskip \noindent
{\bf Subcase 2.1.} The  number of unbalanced circuits in $RY$ or $BY$ is even.

 Let $f_2$ be the $3$-flow in Subcase 1.1. Then  $\supp(f_2) \cup \supp(f_3) = E(G)$. Thus $3f_3 + f_2$ is a $6$-NZF of $(G,\sigma)$.

\medskip\noindent
{\bf Subcase 2.2.} The  number of unbalanced circuits in $RY$ or $BY$ is   equal to one.

  By symmetry, assume that $RY$ has exactly one odd component, say $C_1$. Let $\mathcal{C}=\{C_1, \dots, C_t \}$ be the set of components of $RY$, where each $C_i$ ($i \geq 2$) is balanced, and, with some switching operations, we may assume that the  edges of each $C_i$ ($i \geq 2$) are all  positive. Let $H$ be the signed graph obtained from $(G,\sigma)$ by contracting $\mathcal{C} - C_1$. Then $V(H)$ can be partitioned  into $K$ and $\overline{K}$, where $K=V(C_1)$ and $\overline{K}$ is the set of vertices corresponding to the balanced circuits in $\mathcal{C}$.  For $u\in \overline{K}$, denote the corresponding circuit in $\mathcal{C}$ by $C_u$. Since $G$ is flow-admissible, $H$ remains flow-admissible. Note that $C_1$ is an unbalanced circuit in $H$.

     We consider the following two cases.

\medskip
 {\bf Subcase 2.2.1.} $H$  contains an unbalanced circuit $C'$ that is edge-disjoint from $C_1$.

Since $G$ is cubic,  $C'$ is vertex-disjoint from $C_1$. Thus there is a long barbell $Q$  in $H$ with $P$ as the path connecting $C_1$ and $C'$.  Let $\tau_1$ be the orientation of $Q$ which is a restriction of $\tau$ on $H(Q)$. By Lemma~\ref{TH: 2-to-3}, let $(\tau_1, f'')$ be a $3$-NZF in $Q$. Since $d_Q(u)=2$ or 3 for any $u\in V(Q) - V(C_1)$, $u$ is corresponding to an all-positive circuit $C_u$ in $(G,\sigma)$ with $|\delta_Q (V(C_u))|=2$ or $3$.  Hence by Lemma~\ref{LE: flow-extension} we can extend $f''$ to a 4-flow $f'$  $(G,\sigma)$ with $\bigcup_{u\in V(Q)}E(C_u)\cup E(C_1)\subseteq \supp(f')$. Since  for each $v\in V(H) - V(Q)$, $C_v$ is a balanced circuit in $(G, \sigma)$ , $(G,\sigma)$  admits a  $2$-flow $\phi_v$ with $E(C_v)=\supp(\phi_v)$. Thus $f_4=f'+\sum_{u\in V(H) - V(Q)}\phi_u$ is a $4$-flow of $(G,\sigma)$ with $RY\subseteq \supp(f_4)$. Therefore,  $f_3+ 2 f_4$ is an $8$-NZF  of $(G,\sigma)$.

\medskip  {\bf Subcase 2.2.2.}
$H$  contains no unbalanced circuit  that is edge-disjoint from $C_1$.

In this case,  $H-E(C_1)$ is balanced and thus $G-E(C_1)$ is balanced. With some switching operations we may assume $E_N (G, \sigma) \subseteq E_G(C_1)$.  By Lemma~\ref{HC-cover},  $(G,\sigma)$ has  a $4$-flow $f''$ such that $C_1 \subseteq \supp(f'')$ and every vertex in $\supp(f'')-E(C_1)$ has degree at most $3$ in $H$. By Lemma~\ref{LE: flow-extension}, we can extend $f''$ to a $4$-flow  $f_5$  of $(G,\sigma)$ with $RY\subseteq \supp(f_5)$ in $(G, \sigma)$. Therefore,  $f_3+2f_5$ is an $8$-NZF of $(G,\sigma)$.

 \medskip \noindent
 {\bf Subcase 2.3.} The  number of unbalanced circuits in $RY$ or $BY$ is odd and is at least $3$.

 By symmetry, assume that the  number of unbalanced circuits in $RY$  is odd and is at least $3$.   By Lemma~\ref{flow of g-barbell},  $(G,\sigma)$ has a $5$-flow $f_6$ such that $RY \subseteq \supp(f_6)$ and $E_{f_6 = \pm 4} \subseteq B$. Then $\supp(f_3) \cup \supp(f_6) = E(G)$. Thus $5f_3 + f_6$ is a $10$-NZF of $(G,\sigma)$.
\end{proof}

Next we will prove Corollary ~\ref{COR}.

\medskip \noindent
{\bf Corollary~\ref{COR}} 
Every flow-admissible  bridgeless planar  signed graph admits a nowhere-zero $10$-flow.

\begin{proof}
Let $(G,\sigma)$ be a flow-admissible bridgeless planar signed graph.  Let $\tau$ be an orientation of $(G,\sigma)$. In the following we always assume the flows are under the orientation $\tau$ or its restriction on according subgraphs.

We may assume that the minimum degree of $G$ is at least $3$ otherwise we can suppress all degree $2$ vertices.  We may also assume that $G$ contains no positive loops.

 If $G$ is cubic, then by Theorem~\ref{THM-10-flow} and the $4$-color theorem, $G$ admits a nowhere-zero $10$-flow.  

Suppose that $G$ is not cubic and that $G$ is already embedded in a sphere.  Let $v$ be a vertex with $d_G(v) \geq 4$ and  $t$ be the number of negative loops adjacent to $v$.  First delete the $t$ negative loops and then  blow up  $v$ into a circuit  $C_v$   of length $d_G(v) - 2t$ where each edge of  $C_v$   is positive.   Let $xy$ be an edge  in $C_v$. Replace it with a subdivided edge $u_0u_1u_2\cdots u_{2t+1}$ where $x=u_0$ and $y=u_{2t+1}$  and then replace each $u_iu_{i+1}$ with an unbalanced  digon for each $i =1,3,\dots, 2t-1$ (see Figure ~\ref{FIG: 4}). 
  Let $(G', \sigma')$ be the resulting signed graph obtained from $(G,\sigma)$ by applying the above operations on each vertex in $G$ of degree at least $4$. Then $(G', \sigma')$ is cubic, planar, and  flow-admissible. By Theorem~\ref{THM-10-flow} and the $4$-color theorem, $(G', \sigma')$ admits a nowhere-zero $10$-flow.  Note that  $(G,\sigma)$ can be obtained from $(G',\sigma')$ by contracting  an all-positive subgraph of $(G', \sigma')$. Thus  $(G,\sigma)$ admits a nowhere-zero $10$-flow.
\end{proof}
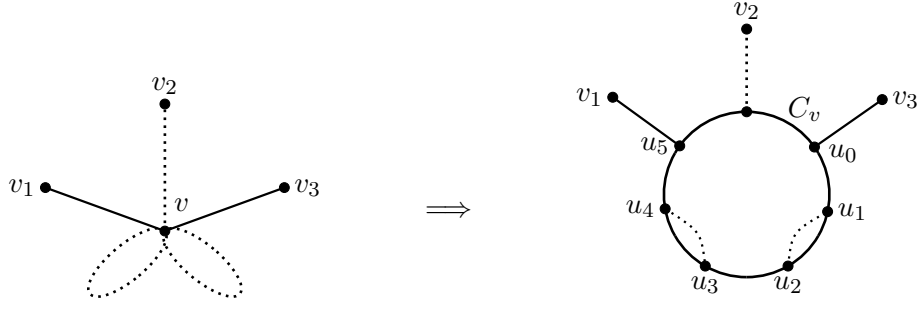
\begin{figure}
\begin{center}
\begin{tikzpicture}[scale=1.3]
\path (0:0)   coordinate (v);  \draw [fill=black] (v) circle (0.05cm);
\path (20:1.3) coordinate (u2); \draw [fill=black] (u2) circle (0.05cm);
\path (160:1.3) coordinate (u3); \draw [fill=black] (u3) circle (0.05cm);
\path (90:1.3) coordinate (u1); \draw [fill=black] (u1) circle (0.05cm);

\draw [dotted,line width=1] (v)--(u1);
\draw [line width=0.85] (u2)--(v);
\draw [line width=0.85] (u3)--(v);
\draw[dotted,line width=1,rotate=310] (0,-0.5)
ellipse (0.18 and 0.5);
\draw[dotted,line width=1,rotate=50] (0,-0.5)
ellipse (0.18 and 0.5);
\node[above]at (30:0.2) {$v$};
\node[right] at (20:1.3) {$v_3$};
\node[above] at (90:1.3) {$v_2$};
\node[left] at (160:1.3) {$v_1$};

\end{tikzpicture}\hspace{1cm}
\begin{tikzpicture}
\path (0:0)   coordinate (v);
\node[above]at (0,1) {$\Longrightarrow$};
\end{tikzpicture}\hspace{1cm}
\begin{tikzpicture}[scale=1.1]
\path (0:0)   coordinate (v);  
\path (-30:1) coordinate (u2); 
\path (90:1) coordinate (u1); \draw [fill=black] (u1) circle (0.06cm);
\path (-150:1) coordinate (u1');

\path (-90:1) coordinate (w);
\path (35:1) coordinate (v1); \draw [fill=black] (v1) circle (0.06cm);
\path (-12:1) coordinate (v2); \draw [fill=black] (v2) circle (0.06cm);
\path (-60:1) coordinate (v3); \draw [fill=black] (v3) circle (0.06cm);
\path (-120:1) coordinate (v4); \draw [fill=black] (v4) circle (0.06cm);
\path (-170:1) coordinate (v5); \draw [fill=black] (v5) circle (0.06cm);
\path (144:1) coordinate (v6); \draw [fill=black] (v6) circle (0.06cm);
\path (144:2) coordinate (u6);   \draw [fill=black] (u6) circle (0.06cm);
\path (90:2) coordinate (u7);  \draw [fill=black] (u7) circle (0.06cm);
\path (35:2) coordinate (u8);  \draw [fill=black] (u8) circle (0.06cm);

\draw [line width=1]  (u2) arc(-30:30:1);
\draw [line width=1]  (u1) arc(90:30:1);
\draw [line width=1]  (u1) arc(90:150:1);
\draw [line width=1]  (u1') arc(210:150:1);
\draw [line width=1]  (w) arc(-90:-30:1);
\draw [line width=1]  (w) arc(-90:-150:1);

\node[right]at (70:1.1) {$C_v$};
\node[left] at (-170:1) {$u_4$};
\node[left,below]at (-120:1) {$u_3$};
\node[below]at (-60:1) {$u_2$};
\node[right] at (-12:1) {$u_1$};
\node[right] at (30:1) {$u_0$};
\node[left] at (144:1) {$u_5$};
\node[left] at (144:2) {$v_1$};
\node[above] at (90:2) {$v_2$};
\node[right] at (35:2) {$v_3$};

\draw [line width=0.85] (u6)--(v6);

\draw [dotted,line width=1] (u1)--(u7);
\draw [line width=0.85] (u8)--(v1);
\draw [dotted,line width=0.85] (v5) .. controls (-0.55,-0.45)..(v4);
\draw [dotted,line width=0.85] (v3) .. controls (0.55,-0.45)..(v2);
\end{tikzpicture}
\end{center}
\caption{\small blowing up of a vertex $v$ with $d(v) = 7$ and $t=2$. Dotted lines are negative edges. }
\label{FIG: 4}
\end{figure}

\section{Proof of Theorem \ref{THM-9-flow}}
\label{SEC:hamilton}

Let's first recall Theorem \ref{THM-9-flow}.

\medskip \noindent
{\bf Theorem \ref{THM-9-flow}}  {\it If $(G,\sigma)$ is a flow-admissible  hamiltonian signed graph, then $(G,\sigma)$ admits a nowhere-zero 8-flow.}

\begin{proof} Let $\tau$ be an orientation of $(G,\sigma)$. In the following we always assume the flows are under the orientation $\tau$ or its restriction on according subgraphs. Let $C_0$ be a hamiltonian  circuit of $G$.  We consider two cases according to whether $C_0$ is balanced or unbalanced.

\medskip \noindent
{\bf
Case 1. }
$C_0$ is balanced.

We may assume that  $C_0$ is all-positive with  some switching operations. It is known that  every ordinary graph with a hamiltonian circuit admits a 4-NZF (See Corollary 3.3.7 \cite{Zhang2002}).  Thus we  may  further  assume  that $(G,\sigma)$ is unbalanced. Hence, by Lemma \ref{Flow-amissible}, $G$ contains at least two negative edges. Clearly, $\langle C_0\rangle_2=(G,\sigma)$.  By Lemma~\ref{2-operation},  $(G, \sigma)$ admits a $\Z_3$-flow $\phi$ such that $E(G) - E(C_0)\subseteq \supp(\phi)$. By Lemma \ref{TH: 3-to-4}, $(G,\sigma)$ admits a  $4$-flow $f_1$ such that $E(G) - E(C_0)\subseteq \supp(\phi)\subseteq \supp(f_1)$.

Since $C_0$ is balanced, $(G,\sigma)$ has a $2$-flow $f_2$ such that  $E(C_0)=\supp(f_2)$. Note that $\supp(f_2)\cup \supp(f_1)=E(G)$. Therefore  $f=2f_1+f_2$ is an 8-NZF of $(G,\sigma)$.

\medskip \noindent
{\bf Case 2.}  $C_0$ is unbalanced.

Since   $C_0$ is unbalanced,  for each edge $e\not \in E(C_0)$, there is a balanced circuit in $C_0+ e$ containing $e$, denoted by $C_e$.
Let $H = \bigtriangleup_{e \notin E(C_0)} C_e$.  Then $H$ admits a $\Z_2$-NZF and  has an even number of negative edges.

If $H$ doesn't contain an unbalanced circuit, then   we may assume that $E(H)$ are all positive  with some switching operations. Thus $E_N(G) \subseteq E(C_0)$ and $(G,\sigma)$ has  a $2$-flow $f_3$ such that  $\supp(f_3) = E(H)$.   By Lemma \ref{HC-cover}, there exists a 4-flow $f_4$ such that $E(C_0)\subseteq \supp(f_4)$.  Since $E(C_0) \cup E(H) = E(G)$, $f_3+2f_4$ is an 8-NZF of $(G,\sigma)$.

Now assume that $H$ contains an unbalanced circuit, say $C_0'$.   Since $H$  admits a $\Z_2$-NZF and  has an even number of negative edges,  by Lemma~\ref{TH: 2-to-3}, $(G,\sigma)$ has a $3$-flow $f_5$  such that $E_{f_5=\pm1}=E(H)$ and $E_{f_5 = \pm 2} \subseteq E(C_0) - E(H)$.

Let $H' = C_0\bigtriangleup C_0'$. Then $H'$ admits a $\Z_2$-NZF and  has an even number of negative edges. Since $C_0\cup C_0'$ is connected, by Lemma~\ref{TH: 2-to-3} again, $(G,\sigma)$ has  a  $3$-flow $f_6$ such that $\supp(f_6) \subseteq E(C_0)\cup E(C_0')$,   $E_{f_6=\pm1}=E(H')$, and $E_{f_6 = \pm 2} \subseteq E(C_0)\cap E(C_0')$.

 Therefore, $3f_5+f_6$ is a $9$-NZF of $(G,\sigma)$.  Since $E_{f_5 = \pm 2} \cap E_{f_6 = \pm 2} = \emptyset$,  $|(3f_5+f_6)(e)| \not = 8$ for each edge $e \in E(G)$.  Thus, $3f_5+f_6$ is  indeed an $8$-NZF of $(G,\sigma)$.
\end{proof}


\begin{thebibliography}{s2}

\bibitem{BM2008}
J.A. Bondy, U.S.R. Murty, Graph Theory, Springer, New York, 2008.

\bibitem{Bouchet1983}
A. Bouchet, Nowhere-zero integral flows on bidirected graph, {\JCTB}~34 (1983) 279--292.

\bibitem{Diestel2010} R. Diestel,  Graph Theory, Fourth edn. Springer-Verlag (2010).









\bibitem{CLLZ2018}
J. Cheng, Y. Lu, R. Luo and C.-Q. Zhang, Signed graphs: from modulo flows to integer-valued flows, \SIAMDM~32 (2018)  956-965.

\bibitem{DLLLZZ}
M. DeVos, J. Li, Y. Lu, R. Luo, C.-Q. Zhang,  and Z. Zhang, Flows on flow-admissible signed graphs, \JCTB~149 (2021) 198-221.

\bibitem{Mac-steffen2015} E.~Ma\v{c}ajova  and M.~\v{S}koviera, Remarks on nowhere-zero flows in signed cubic graphs, \DM~338 (2015) 809-815.

\bibitem{MS-eulerian} E.~Ma\v{c}ajova  and M.~\v{S}koviera,  Nowhere-Zero Flows on Signed Eulerian Graphs, \SIAMDM~31 (2017)  1937-1952.


\bibitem{Schubert2015}
M.~Schubert and E.~Steffen, Nowhere-zero flows on signed regular graphs, \EJC~{48} (2015) 34-47.



\bibitem{West1996} D.B. West, Introduction to Graph Theory, Upper Saddle River, NJ: Prentice Hall, (1996).


\bibitem{Zas1982}  T. Zaslavsky, Signed graphs, \DAM~4(1982) 47-74.


\bibitem{Xu2005} R. Xu, C.-Q. Zhang, On flows in bidirected graphs, \DM~299 (2005) 335--343.

\bibitem{Zhang2002}
C.-Q.~Zhang, Circular flows of nearly eulerian graphs and vertex-splitting, \JGT~{40} (2002) 147-161.

\bibitem{Zyka1987} O. Z\'{y}ka, Nowhere-zero 30-flow on bidirected graphs, Thesis, Charles University, Praha, KAM-DIMATIA Series 87-26, 1987.
\end{thebibliography}
\end{document}